\documentclass[preprint,authoryear,11pt,times]{elsarticle}

\usepackage{amssymb}
 \usepackage{amsthm,amsmath}
 \usepackage{bm}

\journal{Statistics and Probability Letters}

\def\<{\langle}
\def\>{\rangle}
\def\P{\mathbb P}

\def\E{\mathbb E}

\def\F{\mathcal F}

\def\eps{\epsilon}

\def\0{\underline 0}

\def\1{\underline 1}

\def\x{\underline x}

\def\var{\mathbb{V}\mathrm{ar}}
\def\cov{\mathbb{C}\mathrm{ov}}

\def\la{\lambda}

\newcommand{\bel}{\begin{equation}\label}
\newcommand{\ee}{\end{equation}}
\def\tfrac{\frac}
      \newtheorem{theorem}{Theorem}[section]
       
       \newtheorem{cor}[theorem]{Corollary}
       
       \newtheorem{remark}[theorem]{Remark}
\theoremstyle{definition}

\newtheorem{example}{Example}[section]

\begin{document}

\begin{frontmatter}



\title{Double asymptotics for the chi-square statistic}

 \author[osu]{Grzegorz A. Rempa{\l}a}
  \author[pw]{Jacek Weso{\l}owski}
 \address[osu]{Division of Biostatistics and Mathematical Biosciences Institute, The Ohio State University, 43210 Columbus, OH USA }
\address[pw]{Wydzia{\l} Matematyki i Nauk Informacyjnych,  Politechnika Warszawska, Warsaw, Poland}


\begin{abstract}
We consider distributional limit of the Pearson chi-square statistic when the number of classes $m_n$ increases with the sample size $n$ and $n/\sqrt{m_n}\to\la$. Under  mild moment conditions, the limit is Gaussian for $\la=\infty$,  Poisson for finite  $\la>0$, and degenerate for $\la=0$.

\end{abstract}


\end{frontmatter}



\section{Preliminaries} The Pearson chi-square statistic is probably one of the best-known and most important objects of  statistical science and has played a major  role in statistical applications ever since its first appearance in Karl Pearson's work on ``randomness testing"  \citep{pearson1900x}. The
standard test for goodness-of-fit  with the  Pearson chi-square statistic   tacitly assumes that  the support of the discrete distribution of interest is  fixed (whether finite or not) and  unaffected by the sampling process.  However,  this assumption may be unrealistic for   modern  'big-data' problems  which   involve  complex, adaptive  data acquisition processes (see, e.g., \citealt{grotzinger2014habitable} for an  example in astro-biology).  In many such cases the associated   statistical testing problems  may be   more accurately  described in terms    of  triangular arrays of    discrete distributions  whose   finite supports  are dependent upon the collected samples and increase with the samples' size \citep{Pietrzak2016}.   Motivated by 'big-data' applications,  in this note we  establish some asymptotic results for  the Pearson chi-square statistic for  triangular arrays of discrete random variables  for which their  number of  classes $m_n$  grows with the sample size $n$. Specifically,
let $X_{n,k}$, $k=1,\ldots,n$, be iid random variables having the same distribution as $X_n$, where
$$
\P(X_n=i)=p_n(i)>0,\qquad i=1,2,\ldots,m_n<\infty,\qquad n=1,2,\ldots
$$
Recall that the standard Pearson chi-square statistic is defined as
\bel{stachi2}
\chi^2_n=n\sum_{i=1}^{m_n}\,\tfrac{\left(\hat{p}_n(i)-p_n(i)\right)^2}{p_n(i)},
\ee
where the empirical frequencies $\hat{p}_n(i) $ are
$$
\hat{p}_n(i)=n^{-1}\sum_{k=1}^n\,I(X_{n,k}=i),\quad i=1,\ldots,m_n.
$$
As stated above, in what follows  we will be interested in the {\em double asymptotic} analysis   of the weak limit of $\chi^2_n$, that is,  the case when  $m_n\to\infty$ as  $n\to\infty$.

Observe that
 $\chi_n^2$ given in \eqref{stachi2} can be decomposed into a sum of two uncorrelated components as  follows
\bel{deco}
\chi_n^2=n^{-1}\left(U_n+S_n\right)-n,
\ee
where
\bel{U_n}
U_n=\sum_{1\le k\ne l\le n}\,\tfrac{I(X_{n,k}=X_{n,l})}{p_n(X_{n,k})}
\ee
and
\bel{S_n}
S_n=\sum_{k=1}^n\,\tfrac{1}{p_n(X_{n,k})}=\sum_{k=1}^n p_n^{-1}(X_{n,k}).
\ee  The second equality  above introduces notational convention we use throughout.
Note that for fixed $n$ the statistic $S_n$ is simply a sum of iid random variables and $U_n$ is an unnormalized $U$-statistic \citep[see, e.g.,][]{korolyuk2013theory}.  It is routine to check that $$\E\,U_n=n(n-1)\quad\mbox{ and }\quad \E\,S_n=nm_n$$ and consequently  $$\E\,\chi_n^2=m_n-1.$$
Moreover, since we also have  $\cov(U_n,S_n)=0$, it follows that
 $$
\var\,\chi_n^2=n^{-2}(\var\,S_n+\var\,U_n) =n^{-1}[{\var\,{p_n^{-1}(X_n)}+2(n-1)(m_n-1)}].
$$
When $m_n=m$ is a constant then the classical result  \citep[see, e.g.,][chapter 6]{Shao03} implies that   the statistic $\chi_n^2$  asymptotically follows  the $\chi^2$-distribution  with $(m-1)$ degrees of freedom.  Consequently,    when $m$ is large  the standardized statistic  $(\chi^2_n-(m-1))/\sqrt{2(m-1)}$ may be  approximated by the standard normal distribution. However, in  the case when $m_n\to \infty$ as $n\to \infty$ the matters appear to be more subtle  and the above normal approximation may or may not be valid depending upon the asymptotic relation of $m_n$ and $n$, as  described below.   Since  $S_n$ is a sum of iid random variables, the case when $S_n$ contributes to the limit of normalized $\chi^2_n$ may be largely  handled with the standard theory for arrays of iid variables. Consequently, we focus  here on  a seemingly   more interesting case
when the asymptotic influence of $U_n$ dominates over that of $S_n$.
Specifically, throughout the paper we  assume  that as $n,m_n\to \infty$
$$ \leqno(\text{C})\hspace{1.7in} (m_nn)^{-1} \var\,{p_n^{-1}(X_n)}\to 0.$$ Note that (C) implies
$  n^{-1} (S_n-nm_n)/\sqrt{2m_n} \to 0$ in probability
and, in particular, is trivially satisfied when $X_n$ is a uniform random variable on the integer lattice $1,\ldots,m_n$, that is,  when $p_n(i)=m^{-1}_n $ for $i=1\ldots, m_n$. Under condition (C)  we get a rather complete picture of the limiting behavior of $\chi_n^2$. Our main results are presented in  Section 2 where we discuss the  Poissonian and Gaussian asymptotics.  Some examples, relations to asymptotics known in the literature and further discussions are provided in Section 3. The basic tools used in our derivations are listed in the appendix. In   what follows limits are taken as  $n\to \infty$  with $m_n\to \infty$ and $\stackrel{d}\to$ stands for convergence in distribution.

\section{Poissonian and Gaussian asymptotics}
We start with the case when a naive normal approximation for the standardized $\chi^2_n$ statistic fails. Indeed,   as it turns out,  when $m_n$ is asymptotically of order $n^2$,  we have the following Poisson limit theorem for $\chi^2_n$.  \begin{theorem}\label{tha}
Assume that the condition (C) holds,  as well as
\bel{rez}
\tfrac{n}{\sqrt{m_n}}\to\lambda\in(0,\infty).
\ee
Then
\bel{limiP}
\tfrac{\chi_n^2-m_n}{\sqrt{2 m_n}}\stackrel{d}{\to}\tfrac{\sqrt{2}}{\lambda}Z-\tfrac{\lambda}{\sqrt{2}},\qquad Z\sim\mathrm{Pois}\left(\tfrac{\lambda^2}{2}\right)
\ee
\end{theorem}

\begin{proof} 
Due to (C) it suffices to consider the asymptotics of $U_n$ alone. We write
\bel{statist}
\tfrac{U_n-n(n-1)}{n\sqrt{2m_n}}=\tfrac{\sqrt{2m_n}}{n}\,\sum_{k=1}^n\,A_{n,k}-\tfrac{n-1}{\sqrt{2m_n}},
\ee
where $A_{n,1}=0$ and for $k=2,\ldots,n$
\bel{ank}
A_{n,k}=m^{-1}_n\sum_{j=1}^{k-1}\,\tfrac{I(X_{n,j}=X_{n,k})}{p_n(X_{n,j})}={m_n^{-1}p_n^{-1}(X_{n,k})}\sum_{j=1}^{k-1}\,I(X_{n,j}=X_{n,k}).
\ee
The above representation implies that to prove  \eqref{limiP} we need only to show that $\sum_{k=1}^n\,A_{n,k}\stackrel{d}{\to}\mathrm{Pois}\left(\tfrac{\lambda^2}{2}\right)$. To this end we will verify the conditions of  Theorem \ref{BKS} in the appendix,  due to Be\'ska, K\l opotowski and S\l omi\'nski
\citep{beska1982limit}. Denote  $\mathcal{F}_{n,0}=\{\emptyset,\,\Omega\}$ and $\mathcal{F}_{n,k}=\sigma(X_{n,1},\ldots,X_{n,k})$, $k=1,\ldots,n$. Then using the first form of $A_{n,k}$ from \eqref{ank} we see that
\begin{align*}
\max_{1\le k\le n}\,\E(A_{n,k}|\mathcal{F}_{n,k-1})& ={m^{-1}_n}\max_{1\le k\le n}\,\sum_{j=1}^{k-1}\,\E\,\left(\left.\tfrac{I(X_{n,j}=X_{n,k})}{p_n(X_{n,j})}\right|\mathcal{F}_{n,k-1}\right) \\ &=\max_{1\le k\le n}\,\tfrac{k-1}{m_n}=\tfrac{n-1}{m_n}\to 0
\end{align*}
due to \eqref{rez} and thus \eqref{BKS1} holds. Similarly,
\bel{kin}
\sum_{k=1}^n\,\E(A_{n,k}|\mathcal{F}_{n,k-1})=\sum_{k=1}^n\,\tfrac{k-1}{m_n}=\tfrac{n(n-1)}{2m_n}\to \tfrac{\lambda^2}{2}
\ee
and thus \eqref{BKS2} also follows with $\eta=\tfrac{\lambda^2}{2}$. Since $A_{n,k}\ge 0$ the required convergence  in \eqref{BKS3} (for any $\eps>0$) will follow from  convergence of the unconditional moments
\bel{bnd}
\sum_{k=1}^n\,\E\,A_{n,k}I(|A_{n,k}-1|>\eps)\le \eps^{-2}\,\sum_{k=1}^n\,\left(\E\,A_{n,k}^3-2\E\,A_{n,k}^2+\E\,A_{n,k}\right).
\ee

Using the second form of $A_{n,k}$ from \eqref{ank} we see that the conditional distribution of $m_n\,p_n(X_{n,k})\,A_{n,k}$ given $X_{n,k}$ follows a binomial distribution $\mathrm{Binom}(k-1,\,p_n(X_{n,k}))$. Since for $M\sim \mathrm{Binom}(r,p)$ we have $\E\,M=rp$, $\E\,M^2=rp+r(r-1)p^2$ and $\E\,M^3=rp+3r(r-1)p^2+r(r-1)(r-2)p^3$,
we thus obtain
$$
\sum_{k=1}^n\,\E\,A_{n,k}=\tfrac{1}{m_n}\sum_{k=1}^n (k-1)\simeq \tfrac{n^2}{2m_n}\to \tfrac{\lambda^2}{2},
$$
$$
\sum_{k=1}^n\,\E\,A^2_{n,k}=\tfrac{1}{m^2_n}\sum_{k=1}^n \left((k-1)m_n+(k-1)(k-2)\right)\simeq\tfrac{n^2}{2m_n}+\tfrac{n^3}{3m_n^2}\to \tfrac{\lambda^2}{2}.
$$
Similarly,
\begin{align*}
\sum_{k=1}^n\,\E\,A^3_{n,k} &=\tfrac{1}{m^3_n}\sum_{k=1}^n \left((k-1)\E\,{p^{-2}_n(X_n)}+3(k-1)(k-2)m_n+(k-1)(k-2)(k-3)\right)\\
&\simeq \tfrac{n^2}{2m_n^3}\E\,{p^{-2}_n(X_n)}+\tfrac{n^3}{m_n^3}+\tfrac{n^4}{4m_n^3}.
\end{align*}
Note that (C) and \eqref{rez} imply $m_n^{-2}\,\E\,{p_n^{-2}(X_n)}\to 1$ and therefore 
$$
\sum_{k=1}^n\,\E\,A^3_{n,k}\simeq \tfrac{n^2}{2m_n^3}\E\,\tfrac{1}{p^2_n(X_n)}\to \tfrac{\lambda^2}{2}.
$$
Combining the limits of the last three expressions we conclude that the right-hand side of \eqref{bnd} tends to zero and hence \eqref{BKS3} of  Theorem~\ref{BKS} is  also satisfied. The result follows.
\end{proof}

Let us now  consider   the case $\tfrac{n}{\sqrt{m_n}}\to \infty$. As it turns out,  under this condition  the statistic $\chi^2_n$ is asymptotically Gaussian.

\begin{theorem}\label{thb}
Assume that condition (C) is satisfied  and that there exists $\delta>0$ such that
\bel{novnd}
\sup_n \, {m_n^{-(1+\delta)}}\,\E\,{p_n^{-(1+\delta)}(X_n)}<\infty
\ee
as  well as
\bel{reff}
\tfrac{n}{\sqrt{m_n}}\to \infty.
\ee

Then
\bel{limi}
\tfrac{\chi_n^2-m_n}{\sqrt{2m_n}}\stackrel{d}{\to}
N,\quad N\sim \mathrm{Norm}(0,1).
\ee
\end{theorem}
\begin{remark} Note that under  (C)  the  conditions \eqref{novnd} (with $\delta=1$) and \eqref{reff} are implied by the  condition $n/m_n\to\la\in(0,\infty)$.
\end{remark}
\begin{proof}
As in Theorem~\ref{tha}, under our assumption (C)  it suffices to show convergence in distribution to $N\sim \mathrm{Norm}(0,1)$ of the normalized $U_n$ variable
$$
\tfrac{U_n-n(n-1)}{\sqrt{n(n-1)2(m_n-1)}}=\sum_{k=1}^n\,Y_{n,k},
$$
where
\bel{repr}
Y_{n,k}=\tfrac{\sqrt{2}}{\sqrt{n(n-1)(m_n-1)}}\,\sum_{j=1}^{k-1}\,\left(\tfrac{I(X_{n,j}=X_{n,k})}{p_n(X_{n,j})}-1\right)=\tfrac{\sqrt{2}\,B_{n,k}}{\sqrt{n(n-1)(m_n-1)}}
\ee and the last equality defines  $B_{n,k}$.
Since $\E(I(X_{n,k}=X_{n,j})|\F_{n,k-1})=p_n(X_{n,j})$ for any $j=1,\ldots,k-1$, it follows that $\E(Y_{n,k}|\F_{n,k-1})=0$. Consequently, $(Y_{n,k},\,\F_{n,k})_{k=1,\ldots,n}$ are martingale differences. Therefore, to prove \eqref{limi} we may use the Lyapounov version of the CLT  for martingale differences  (see Theorem~\ref{mclt} in the appendix).

Due to \eqref{repr} we have
\begin{align*}
\E(B_{n,k}^2|\F_{n,k-1}) &=\sum_{j=1}^{k-1}\,\tfrac{\var(I(X_n=X_{n,j})|\F_{n,k-1})}{p_n^2(X_{n,j})} \\
&+\sum_{1\le i\ne j\le k-1}\,\tfrac{\cov(I(X_n=X_{n,i}),\,I(X_n=X_{n,j})|\F_{n,k-1})}{p_n(X_{n,i})p_n(X_{n,j})}.
\end{align*}
Since
$
\var(I(X_n=X_{n,j})|\F_{n,k-1})=p_n(X_{n,j})(1-p_n(X_{n,j}))
$
and
$$
\cov(I(X_n=X_{n,i}),\,I(X_n=X_{n,j})|\F_{n,k-1})=
I(X_{n,i}=X_{n,j})p_n(X_{n,i})-p_n(X_{n,i})p_n(X_{n,j})
$$
we obtain
$$
\E(B_{n,k}^2|\F_{n,k-1})=\sum_{j=1}^{k-1}\,\left({p^{-1}_n(X_{n,j})}-1\right)+\sum_{1\le i\ne j\le k-1}\,\left(\tfrac{I(X_{n,i}=X_{n,j})}{p_n(X_{n,i})}-1\right).
$$
Consequently, \eqref{convar} is equivalent to
\bel{equi}
\tfrac{\sum_{k=1}^n\,\sum_{j=1}^{k-1}\,\left({p^{-1}_n(X_{n,j})}-m_n\right)}{\tfrac{n(n-1)}{2}(m_n-1)}+\tfrac{\sum_{k=1}^n\,\sum_{1\le i\ne j\le k-1}\,\left(\tfrac{I(X_{n,i}=X_{n,j})}{p_n(X_{n,i})}-1\right)}{\tfrac{n(n-1)}{2}(m_n-1)}\stackrel{\P}{\to}0.
\ee
To show the above,  we separately consider moments of the  summands on  the left-hand side of \eqref{equi}.
For the first one, note that
\begin{align*}
\sum_{k=1}^n\,\sum_{j=1}^{k-1}\left({p^{-1}_n(X_{n,j})}-m_n\right) &=\sum_{j=1}^{n-1}\,(n-j)\left({p^{-1}_n(X_{n,j})}-m_n\right)\\
&\stackrel{d}{=}\sum_{j=1}^{n-1}\,j\left({p^{-1}_n(X_{n,j})}-m_n\right)
\end{align*} where the last equality denotes the distributional equality of random variables.
Therefore, using  inequality \eqref{bur1} given in the appendix,  we get (possibly with different  universal constants $C$ from line  to line)
\begin{align*}
& \E\left|\tfrac{\sum_{k=1}^n\,\sum_{j=1}^{k-1}\left({p^{-1}_n(X_{n,j})}-m_n\right)}{\tfrac{n(n-1)}{2}(m_n-1)}\right|^{1+\delta}
\le C \tfrac{\E\,\left|\sum_{j=1}^{n-1}\,j\left({p^{-1}_n(X_{n,j})}-m_n\right)\right|^{1+\delta}}{n^{2+2\delta}m_n^{1+\delta}}\\
&\le C\tfrac{\E\,\left|{p^{-1}_n(X_{n,j})}-m_n\right|^{1+\delta}\,n^{\tfrac{\delta-1}{2}\vee 0}\,\sum_{j=1}^{n-1}\,j^{1+\delta}}{n^{2+2\delta}m_n^{1+\delta}} \\ & \le C\tfrac{\E\,\left|{p^{-1}_n(X_{n,j})}-m_n\right|^{1+\delta}\,n^{\tfrac{3(1+\delta)}{2}\vee(2+\delta)}}{n^{2+2\delta}m^{1+\delta}_n}  \le C\tfrac{\E\,\left|{p^{-1}_n(X_{n,j})}-m_n\right|^{1+\delta}}{n^{\tfrac{1+\delta}{2}\wedge \delta}m_n^{1+\delta}}.
\end{align*}
In view of  this and the elementary inequality  $|a+b|^p\le C(|a|^p+|b|^p)$ valid for any $p>0$ and any real $a,b$  we have for some constants $C_1,C_2$
$$
\E\left|\tfrac{\sum_{k=1}^n\,\sum_{j=1}^{k-1}\left({p^{-1}_n(X_{n,j})}-m_n\right)}{\tfrac{n(n-1)}{2}(m_n-1)}\right|^{1+\delta}
\le \tfrac{C_1}{n^{\tfrac{1+\delta}{2}\wedge \delta}}\,\tfrac{\E\,p_n^{-(1+\delta)}(X_n)}{m_n^{1+\delta}}+\tfrac{C_2}{n^{\tfrac{1+\delta}{2}\wedge \delta}}\to 0.
$$
For the numerator of the second part on the left hand side of \eqref{equi} we may write
$$
\sum_{k=1}^n\,\sum_{1\le i\ne j\le k-1}\,\left(\tfrac{I(X_{n,i}=X_{n,j})}{p_n(X_{n,i})}-1\right)=2\sum_{1\le i< j\le n-1}\,(n-j)\left(\tfrac{I(X_{n,i}=X_{n,j})}{p_n(X_{n,i})}-1\right).
$$
Moreover,
\begin{align*}
&\E\left(\sum_{1\le i< j\le n-1}\,(n-j)\left(\tfrac{I(X_{n,i} =X_{n,j})}{p_n(X_{n,i})}-1\right)\right)^2\\ &=\sum_{1\le i<j\le n-1}\,(n-j)^2\,\E\left(\tfrac{I(X_{n,i}=X_{n,j})}{p_n(X_{n,i})}-1\right)^2,
\end{align*}
since  the expectations of the other terms resulting from squaring the  large-bracketed first expression above  are equal  to zero. Consequently
\begin{align*}
\E\left(\sum_{1\le i< j\le n-1}\,(n-j)\left(\tfrac{I(X_{n,i}=X_{n,j})}{p_n(X_{n,i})}-1\right)\right)^2 & =(m_n-1)\sum_{1\le i<j\le n-1}\,(n-j)^2\\
&\le C\,m_n n^4
\end{align*}
and thus for the squared expectation of the  second term in \eqref{equi} we get
$$
\E\,\left(\tfrac{\sum_{k=1}^n\,\sum_{1\le i\ne j\le k-1}\,\left(\tfrac{I(X_{n,i}=X_{n,j})}{p_n(X_{n,i})}-1\right)}{\tfrac{n(n-1)}{2}(m_n-1)}\right)^2\le C\, m^{-1}_n\to 0. $$

Note that here we used the fact that $m_n\to \infty$.
 To finish the proof we only need to show \eqref{4m}.
Again we will rely on the representation of $Y_{n,k}$ given in \eqref{repr}. Note that
\begin{align*}
&\E\,\left|Y_{n,k}\right|^{2+\delta}\\ &
\le C n^{-(2+\delta)}m_n^{-(1+\tfrac{\delta}{2})}\,\E\,\left({p_n^{-(2+\delta)}(X_{n,k})}
\,\left|\sum_{j=1}^{k-1}\,(I(X_{n,j}=X_{n,k})-p_n(X_{n,k}))\right|^{2+\delta}\right).
\end{align*}
Since $I(X_{n,j}=X_{n,k})-p_n(X_{n,k})$, $j=1,\ldots,k-1$, are conditionally iid given $X_{n,k}$ and
$$\E((I(X_{n,j}=X_{n,k})-p_n(X_{n,k}))|X_{n,k})=0$$
then by conditioning with respect to $X_{n,k}$ and applying  Rosenthal's inequality (see  \eqref{ros} in the appendix) to the  conditional moment of the sum we obtain
\begin{align}\label{niero}
&\sum_{k=1}^n\,\E\,\left|Y_{n,k}\right|^{2+\delta}\nonumber \\ &\le \tfrac{C}{n^{2+\delta}m_n^{1+\tfrac{\delta}{2}}}\,\sum_{k=1}^n\,\E\,\left({p_n^{-(2+\delta)}(X_n)}\left((k-1)p_n(X_n)+[(k-1)p_n(X_n)]^{1+\tfrac{\delta}{2}}\right)\right)\nonumber \\
&\le C\left(n^{-\delta}m_n^{-(1+\tfrac{\delta}{2})}\,\E\,{p_n^{-(1+\delta)}(X_n)}+n^{-\tfrac{\delta}{2}}\,m_n^{-(1+\tfrac{\delta}{2})}\,
\E\,p_n^{-(1+\tfrac{\delta}{2})}(X_n)\right).
\end{align}
By virtue of the Schwartz inequality we obtain that
\begin{align*}\label{sch}
n^{-\tfrac{\delta}{2}}\,m_n^{-(1+\tfrac{\delta}{2})}\,
\E\,{p_n^{-(1+\tfrac{\delta}{2})}(X_n)} & =n^{-\tfrac{\delta}{2}}\,m_n^{-(1+\tfrac{\delta}{2})}\,
\E\,{p^{-\frac{1}{2}}_n(X_n)}\, {p_n^{-\tfrac{1+\delta}{2}}(X_n)} \\ & \le n^{-\tfrac{\delta}{2}}\,\sqrt{m_n^{-(1+\delta)}\,
\E\,p_n^{-(1+\delta)}(X_n)}\to 0
\end{align*} in view of  \eqref{novnd}.
Therefore, it only suffices to show that the first term  in the last expression  in \eqref{niero} converges to zero. But this follows  due to \eqref{novnd} and \eqref{reff},  since
$$
\tfrac{\E\,{p_n^{-(1+\delta)}(X_n)}}{n^{\delta}m_n^{1+\tfrac{\delta}{2}}}\,=\left(\tfrac{\sqrt{m_n}}{n}\right)^{\delta} \tfrac{\E\,p_n^{-(1+\delta)}(X_n)}{m_n^{1+\delta}}\to 0.
$$
 \end{proof}

 \section{Discussion}
 We will now  illustrate the results of the previous section with some examples as well as put   them in a  broader context of earlier work by others.
 For the sake of completeness,  we  first note
 \begin{remark}\label{rem1}
{\bf The case $\bm\la=\bm 0$.} Consider  $\tfrac{n}{\sqrt{m_n}}\to 0$. Then the last part of  the right hand side of \eqref{statist} converges to zero and we are left with the sum of non-negative random variables which satisfies $$\tfrac{2\sqrt{m_n}}{n}\,\sum_{k=1}^n\,A_{n,k}\stackrel{\P}\to 0.$$ To see the above,   it  suffices to consider the convergence of the first moments. To this end  note that
$$
\tfrac{2\sqrt{m_n}}{n}\,\sum_{k=1}^n\,\E\,A_{n,k}=\tfrac{2\sqrt{m_n}}{n}\,\sum_{k=1}^n\,\tfrac{k-1}{m_n}=\tfrac{n-1}{\sqrt{m_n}}\to 0.
$$
\end{remark}
The simple illustration of   Theorem~\ref{thb} is  as follows.
 \begin{example}\label{ex}
 Let $\alpha\in [0,1)$ and set   $p_n(i)=(C_\alpha i^{\alpha})^{-1}$ for $i=1,\ldots,m_n$.   Here  $C_\alpha=\sum_{i=1}^{m_n} i^{-\alpha}\simeq m_n^{1-\alpha}/(1-\alpha)$  in view of  the  general formula 
 \begin{equation}\label{pwr}
 \sum_{i=1}^{m_n} i^{\beta} \simeq m_n^{\beta+1}/(\beta+1)\quad \text{for}\quad   \beta>-1.
 \end{equation}
   Note that for $0<\alpha<1$  the condition   (C) is equivalent to
 \begin{equation}\label{c_cond}
 n/m_n\to \infty
 \end{equation} and implies   \eqref{reff}. Applying \eqref{pwr} again we see that for any $\delta>0$
 $$ \tfrac{\E\,p_n^{-(1+\delta)}(X_n)}{m_n^{1+\delta}}=\tfrac{C_\alpha^\delta\sum_{i=1}^{m_n}   i^{\alpha\delta} }{m_n^{1+\delta}}\simeq \tfrac{m_n^{(1-\alpha)\delta} m_n^{1+\alpha\delta}}{(1-\alpha)^\delta(1+\alpha\delta) m_n^{1+\delta}}=(1-\alpha)^{-\delta}(1+\alpha\delta)^{-1}<\infty$$  and therefore  \eqref{novnd}  is also satisfied.  Hence, the conclusion of Theorem~\ref{thb} holds true under \eqref{c_cond} for $0<\alpha<1$.   \end{example}

 Note that in the above example the assumption \eqref{rez} of Theorem~\ref{tha} cannot be satisfied for $0<\alpha<1$ (see \eqref{c_cond}) but can hold  for  $\alpha=0$, that is, when the distribution is uniform.  We remark that in our present setting such distribution is  of interest, for instance,  when testing for signal-noise threshold in data with large number of support points  \citep{Pietrzak2016}. Combining  the results of  Theorems~\ref{tha} and \ref{thb} and Remark~\ref{rem1} one obtains    the following.
\begin{cor}[\bf Asymptotics of $\bm \chi^2_n$ for uniform distribution]\label{cor} Assume that $
p_n(i)=m^{-1}_n$ for $i=1,2,\ldots,m_n$ and $n=1,2,\ldots $  as well as $$n/\sqrt{m_n}\to \la.$$Then

$$\tfrac{\chi_n^2-m_n}{\sqrt{2m_n}}\stackrel{d}{\to}\begin{cases} 0 &\text{when $\la=0$}, \\
      \tfrac{\sqrt{2}}{\lambda}Z-\tfrac{\lambda}{\sqrt{2}},\quad Z\sim\mathrm{Pois} \left(\tfrac{\lambda^2}{2}\right) & \text{when $\la\in(0,\infty)$ }, \\
      N\sim \mathrm{Norm}(0,1)& \text{when $\la=\infty$}.
\end{cases} $$ \qed
\end{cor}

We note that the asymptotic distribution of  ${\chi}_n^2$ when both $n$ and $m_n$ tend to infinity has  been considered by several authors, typically in the context of asymptotics of families of goodness-of-fit statistics related to different  divergence distances. Some of these results considered also the asymptotic behavior of such statistics not only under the null hypothesis (as we did here)  but also under simple alternatives and hence are, in that sense, more general. However, when applied to the chi-square statistic under the null hypothesis they appear to be special cases of our theorems  in Section~2.  We  briefly review below some  of the most relevant results.

 \cite{Tu54,Tu56} proved asymptotic normality of ${\chi}^2_n$ under the assumption $\min_{1\le i\le m_n}\,np_n(i)\to \infty$ which in the case of the uniform distribution is equivalent to $n/m_n\to\infty$, a condition obviously stronger than $n/\sqrt{m_n}\to\infty$ we use (see Corollary~\ref{cor}).

 \cite{steck1957limit} generalized these results on normal asymptotics assuming among other conditions that $\inf_n\,n/m_n>0$ which again is stronger than $n/\sqrt{m_n}\to\infty$.  He also  obtained the Poissonian and degenerate limit in the case of uniform distribution, in  agreement with the first two cases in our Corollary~\ref{cor}.
The main result of \cite{holst1972asymptotic} for the chi-square statistic gives normal asymptotics under the regime $n/m_n\to\lambda\in(0,\infty)$ and $\max_{1\le j\le n}\,p_n(j)<\beta/n$ which also is stronger than our assumptions.  In  the uniform case  under this regime the result was  proved earlier in \cite{harris1971distribution}.
The main result of \cite{morris1975central} for the chi-square statistics gives asymptotic normality under $n\min_{1\le j\le n}\,p_n(j)>\eps>0$ for all $n\ge 1$, $\max_{1\le j\le n}\,p_n(j)\,\to 0$ and the "uniform asymptotically negligible" condition of the form ${\max_{1\le i\le m_n}\,\sigma_n^2(i)}/{s_n^2}\to 0$, where $\sigma^2_n(i)=2+\tfrac{(1-m_np_n(i))^2}{np_n(i)}$, $i=1,\ldots,m_n$, and $s_n^2=\sum_{i=1}^{m_n}\,\sigma_n^2(i)$. In the case of the uniform distribution it gives asymptotic normality of ${\chi}^2_n$ under the condition $n/m_n>\eps>0$, the result apparently weaker than the third part of Corollary 3.2.

Following  the paper  of \cite{cressie1984multinomial} introducing the family of power divergence statistics (of which the chi-square statistic is a member), much effort was directed at proving asymptotic normality for wider families of divergence distances as well as for more than one multinomial independent sample, see e.g. \cite{menendez1998asymptotic,parper2002} (in both papers the authors  considered the  regime $n/m_n\to \lambda\in(0,\infty)$) and \cite{inglot1991asymptotics}, \cite{morales2003asymptotic} (in  both papers the authors considered the regime $m_n^{1+\beta}\log^2(n)/n\to 0$ and $m_n^{\beta}\min_{1\le j\le n}\,p_n(j)>c>0$ for some $\beta\ge 1$) or \cite{Pietrzak2016} (with the regime $n/m_n\to\infty$). Note that for the  asymptotic normality results all these regimes are again more stringent than what we  consider here.

Finally, for completeness, we briefly address one of the scenarios  when condition (C) does not hold.
\begin{remark}
Note that  if $\tfrac{m_nn}{\var\,p_n^{-1}(X_n)}\to 0$ then the asymptotic behavior of standardized $\chi^2_n$ is the same as that of $Z_n=\sum_{k=1}^n\,Y_{n,k}$,  where $$Y_{n,k}=\tfrac{p_n^{-1}(X_{n,k})-m_n}{\sqrt{n\var\,p_n^{-1}(X_n)}},\quad k=1,\ldots,n.$$ Since for any fixed $n\ge 1$ random variables $Y_{n,k}$, $k=1,\ldots,n$, are iid (zero mean) and $\var\,Y_{n,k}=n^{-1}$ it follows that $\{Y_{n,k},\,k=1,\ldots,n\}_{n\ge 1}$ is an infinitesimal array. Therefore classical CLT  for row-wise iid triangular arrays  \citep[cf., e.g.,][chapter 1]{Shao03} applies.  Note also that the remaining case when  $\tfrac{m_nn}{\var\,p_n^{-1}(X_n)}\to \la\in(0,\infty)$ appears more complicated and  requires   a different approach.

\end{remark}


\section*{Acknowledgements} The research was conducted when the second author was visiting The Mathematical Biosciences Institute at OSU. Both authors thank the Institute for its logistical support and funding through US NSF grant DMS-1440386. The research was also partially funded by  US NIH grant R01CA-152158 and US NSF grant DMS-1318886. The authors  wish to gratefully acknowledge helpful comments made by  the referee and the associate editor on the early version of the manuscript.
\appendix
\newtheorem{athm}{Theorem}[section]

\renewcommand{\theathm}{A.\arabic{athm}}

\section{Limit Theorems}
Below, for convenience of the readers,  we recall some results which are used in the proofs.
The first one  is found in \cite{beska1982limit} and the second one is a version of the martingale CLT \citep[see, e.g.,][]{hh80}.
\begin{athm}[\bf Poissonian conditional limit theorem]\label{BKS}
Let $\{Z_{n,k},\ k=1,\ldots,n;\,n\ge 1\}$ be a double sequence of non-negative random variables adapted to a row-wise increasing double sequence of $\sigma$-fields $\{\mathcal{G}_{n,k-1},\,k=1,\ldots,n;\,n\ge 1\}$. If for $n\to\infty$
\bel{BKS1}
\max_{1\le k\le n}\,\E(Z_{n,k}|\mathcal{G}_{n,k-1})\stackrel{\P}{\to}0,
\ee
\bel{BKS2}
\sum_{k=1}^n\,\E(Z_{n,k}|\mathcal{G}_{n,k-1})\stackrel{\P}{\to} \eta>0,
\ee
and for any $\eps>0$
\bel{BKS3}
\sum_{k=1}^n\,\E(Z_{n,k}I(|Z_{n,k}-1|>\eps)|\mathcal{G}_{n,k-1})\stackrel{\P}{\to} 0,
\ee
then
$\sum_{k=1}^n\,Z_{n,k}\stackrel{d}{\to}Z$, where $Z\sim \mathrm{Pois}(\eta)$ is a Poisson  random variable.
\end{athm}

\begin{athm}[\bf Lyapunov-type martingale CLT]\label{mclt}
Let $\{(Z_{n,k},\mathcal{F}_{n,k})\,k=1,\ldots,n;\,n\ge 1\}$ be a double sequence of martingale differences.
If
\bel{convar}
\sum_{k=1}^n\,\E\left(Y_{n,k}^2|\mathcal{F}_{n,k-1}\right)\stackrel{\P}{\to}1
\ee
and for some $\delta>0$
\bel{4m}
\sum_{k=1}^n\,\E\,Y_{n,k}^{2+\delta}\to 0.
\ee
then
$\sum_{k=1}^n\,Z_{n,k}\stackrel{d}{\to}N$, where $N\sim\mathrm{Norm}(0,1)$ is  a standard normal  random variable.
\end{athm}
\section{Moment Inequalities} The following moment inequalities  are used in Section~2.
\subsection*{\bf Rosenthal inequality} \cite{rosenthal1970subspaces}.  If $X_1,\ldots,X_n$ are independent and centered random variables such that $\E|X_i|^r<\infty$, $i=1,\ldots,n$ and $r>2$ then
\begin{align}\label{ros}
\E\left|\sum_{i=1}^n\,X_i\right|^r &\le C_r\max\left\{\sum_{i=1}^n\,\E|X_i|^r,\;\left(\sum_{i=1}^n\,\E\,X_i^2\right)^{\tfrac{r}{2}}\right\} \nonumber\\
&\le C_r\left(\sum_{i=1}^n\,\E|X_i|^r+\left(\sum_{i=1}^n\,\E\,X_i^2\right)^{\tfrac{r}{2}}\right).
\end{align}
\subsection*{\bf MZ-BE inequality} \cite{marcinkiewicz1937quelques} for $r\ge 2$,  \cite{von1965inequalities} for $1\le r\le 2$. If $X_1,\ldots,X_n$ are independent and centered random variables such that $\E|X_i|^r<\infty$, $i=1,\ldots,n$ then for $r>1$
\bel{bur1}
\E\left|\sum_{i=1}^n\,X_i\right|^r\le C_r\,n^{r_\ast}\,\sum_{i=1}^n\,\E|X_i|^r,
\ee
where $r_\ast=0\vee \left(\tfrac{r}{2}-1\right)$.


\end{document}